\documentclass[12pt]{article}
\usepackage{amsmath,amssymb,amsthm}
\usepackage[margin=1.25in]{geometry}
\usepackage[bookmarksnumbered]{hyperref}
\usepackage{enumitem,theoremref,verbatim,cancel, breqn, tikz}

\usetikzlibrary{matrix, cd}

\title{Single site factors of Gibbs measures}
\author{Mark Piraino \thanks{Department of Mathematics and Statistics, University of Victoria.}}
\date{\today}

\theoremstyle{definition}
\newtheorem{theorem}{Theorem}[section]
\newtheorem{lemma}[theorem]{Lemma}
\newtheorem{example}[theorem]{Example}
\newtheorem{corollary}[theorem]{Corollary}
\newtheorem{proposition}[theorem]{Proposition}
\newtheorem{definition}[theorem]{Definition}
\newtheorem{procedure}[theorem]{Procedure}

\newtheorem{remark}{Remark}

\newtheorem*{acknowledgements}{Acknowledgements}

\newcommand{\bmat}[1]{\begin{bmatrix} #1 \end{bmatrix}}

\newcommand{\inn}[1]{\left\langle #1 \right\rangle}
\newcommand{\set}[1]{\left\{ #1 \right\}}
\newcommand{\abs}[1]{\left| #1 \right|}

\newcommand{\norm}[1]{\left \| #1 \right \|}

\newcommand{\ol}[1]{\overline{#1}}

\def\[#1\]{\begin{align*}#1\end{align*}}

\DeclareMathOperator{\spn}{span}
\DeclareMathOperator{\diam}{diam}

\DeclareMathOperator{\var}{var}

\newcommand{\N}{\mathbb{N}}

\newcommand{\R}{\mathbb{R}}

\def\A{\mathcal{A}}

\def\L{\mathcal{L}}      
\def\c{\mathcal{C}}    
\def\S{\Sigma}

\newcommand{\e}{\varepsilon}

\begin{document}

\maketitle

\begin{abstract}
	It has been an open problem to identify classes of Gibbs measures less regular then H\"older continuous on the full shift which are closed under factor maps. In this article we show that in fact all of the classical uniqueness regimes (Bowen, Walters, and H\"older) from thermodynamic formalism are closed under factor maps between full shifts. In fact we show more generally that the classical uniqueness regimes are closed under factors between shifts of finite type provided the factor map satisfies a suitable mixing in fibers condition.
\end{abstract}

\section{Introduction}
Hidden Markov measures are of great interest in many areas of science, both pure and applied. It is well known that a hidden Markov measure can fail to be Markov. Our goal here is to study a generalization of hidden Markov measures, single site factors of Gibbs measures. These measures have attracted a significant amount attention (\cite{MR2853610}, \cite{Yoo2010}, \cite{chazottes2003projection}, \cite{chazottes2009preservation}, \cite{johansson_öberg_pollicott_2017}, \cite{kempton2011factors}). Broadly speaking there are two main questions: are these measures Gibbs for some continuous potential (and what is its continuity rate) and what classes of measures are preserved by single site factors. We will focus on the second question.

Let us recall the definition of a single site factor map. Suppose that $\S$ and $\ol{\S}$ are two alphabets and $\pi:\S \to \ol{\S}$ and $\S_{A}^{+}\subseteq \S^{\N}$ the shift of finite type over $\S$ determined by the matrix $A$ we define a map from $\S_{A}^{+}\to Y \subseteq\ol{\S}^{\N}$ which we again call $\pi$ by $\pi[(x_{i})_{i=0}^{\infty}]= (\pi(x_{i}))_{i=0}^{\infty}$. This map is continuous and intertwines the shift maps. Thus given a shift invariant measure $\mu$ on $\S^{\N}$ we can define $\pi_{\ast}\mu$ on $\ol{\S}^{\N}$ as the pushforward of $\mu$ under $\pi$. That is, $\pi_{\ast}\mu(A) = \mu(\pi^{-1}A)$ for all $A$ Borel measurable. As $\pi$ interwines the shift maps it follows that $\pi_{\ast}\mu$ is shift invariant.

When the shift of finite type $\S_{A}^{+}$ is a full shift it is known that factors of Markov measures have H\"older continuous $g$ functions \cite{Yoo2010}. However when the shift $\S_{A}^{+}$ has excluded words this is no longer true see for instance \cite[example 4]{HolderGibbsStates} or \cite[example 4.2]{kempton2011factors}. However by imposing conditions on the factor map $\pi$ to ensure that the fibers of $\pi$ are ``topologically mixing'' in a certain sense the result can be recovered see for instance \cite{chazottes2003projection}, \cite{Yoo2010}. The goal of this paper is to prove the analogous results for Gibbs measures associated to more general potentials. 

\begin{definition}\label{defofCUR}
	Let $\varphi:\S_{A}^{+}\to \R$ be a function. Recall that
	\[ \var_{n}\varphi = \sup\set{\abs{\varphi(x)-\varphi(y)}:x_{i}=y_{i} \text{ for }0\leq i \leq n-1}. \]
	We say that $\varphi$ is \emph{H\"older} if there exists a constant $\abs{\varphi}_{\theta}$ and $0<\theta<1$ such that
	\[ \var_{n}\varphi \leq \abs{\varphi}_{\theta}\theta^{n} \]
	for all $n \geq 0$ (that is $\varphi$ is H\"older in the $2^{-n}$ metric). We say that $\varphi$ is \emph{Walters} if 
	\[ \sup_{n \geq 1}\var_{n+k}S_{n}\varphi\xrightarrow{k \to \infty}0 \]
	where $S_{n}\varphi(x)=\sum_{i=0}^{n-1}\varphi(\sigma^{i}x)$. We say that $\varphi$ is \emph{Bowen} if $\varphi$ is continuous and there exists a constant $K$ such that
	\[ \sup_{n\geq 1}\var_{n}S_{n}\varphi\leq K. \]
\end{definition}

We will refer to these as the classical uniqueness regimes. It can be seen that 
\[ \text{H\"older}\subset \text{Walters} \subset \text{Bowen}. \]

\begin{remark}
	Another common class of potentials are those potentials which have \emph{summable variations}. That is, potentials for which
	\[ \sum_{n=0}^{\infty}\var_{n}\varphi < \infty. \]
	It can be shown that 
	\[ \text{summable variations} \subset \text{Walters} \]
	and in some sense summable variations is simply a verifiable condition which implies the Walters property. 
\end{remark}

For all of these classes of potentials Gibbs measures exist and are unique. That is, if $\S_{A}^{+}$ is a topologically mixing shift of finite type and $\varphi:\S_{A}^{+}\to \R$ is Bowen then there exists a unique shift invariant measure $\mu_{\varphi}$ for which there exist constants $C>0$ and $P$ such that
\begin{equation}\label{eq:scalarGibbinequality}
C^{-1}\leq \frac{\mu_{\varphi}([x_{0}\cdots x_{n-1}])}{e^{-nP+S_{n}\varphi(x)}}\leq C 
\end{equation}
for all $x \in \Sigma_{A}^{+}$ and $n>0$. The constant $P$ is called the pressure and is alternatively characterized by the equation
\[ P=\sup\set{h_{\mu}(\sigma)+\int \varphi d \mu : \mu \text{ is shift invariant}}. \]
Moreover the measure $\mu_{\varphi}$ is the unique measure which achieves this supremum. The following definition appears in \cite{Yoo2010}.
\begin{definition}\label{fibermixingdef}
	We say that $\pi$ is fiber-wise sub-positive mixing if there exists an $N$ such that for any word $b_{0}\cdots b_{N}$ admissible in $Y$ and words $u_{0}\cdots u_{N}, w_{0}\cdots w_{N}$ such that $\pi(u_{0}\cdots u_{N})=\pi(w_{0}\cdots w_{N})=b_{0}\cdots b_{N}$ there exists a word $a_{0}\cdots a_{N}$ admissible in $\S_{A}^{+}$ with $\pi(a_{0}\cdots a_{N})=b_{0}\cdots b_{N}$ and $a_{0}=u_{0}$, $a_{N}=w_{N}$.
\end{definition}

The goal of this article is then to prove the following theorem.

\begin{theorem}\label{mainthm}
	Suppose that $\Sigma$ is a finite alphabet, $\Sigma_{A}^{+}\subseteq \S^{\N}$ is a topologically mixing shift of finite type, $\mu_{\varphi}$ a Gibbs measure for $\varphi$ and $\pi$ a fiber-wise sub-positive mixing factor map. If $\varphi$ is Bowen (respectively Walters, H\"older) then $\pi_{\ast}\mu_{\varphi}$ is the Gibbs state for a potential which is Bowen (respectively Walters, H\"older).
\end{theorem}

The proof also yields the following.
\begin{corollary}\label{maincor}
	Suppose that $\Sigma$ is a finite alphabet, $\Sigma_{A}^{+}\subseteq \S^{\N}$ is a topologically mixing shift of finite type, $\mu_{g}$ a $g$-measure for $g$ and $\pi$ a fiber-wise sub-positive mixing factor map. If $\log g$ is Bowen (respectively Walters, H\"older) then the logarithm of the $g$-function for $\pi_{\ast}\mu_{g}$ is Bowen (respectively Walters, H\"older).
\end{corollary}

In \cite{MR2853610} the problem of finding a class of g measures more general then H\"older continuous which is preserved under single site transformations on the full shift over finite many symbols was posed. In addition it was conjectured that the class of g measures with square summable log variations might be such a class, however it was shown in \cite{johansson_öberg_pollicott_2017} that this is not the case. Theorem \ref{mainthm} can be seen as a strong answer to this question, which holds beyond the special case of factor maps between full shifts.

Recall Gibbs measures are typically constructed using the eigendata of a Ruelle operator. Recall that given a continuous function $\varphi:X \to \R$ (which is often called a \emph{potential}) we define the \emph{Ruelle operator} (sometimes referred to as the \emph{transfer operator}) for $\varphi$, $L_{\varphi}:C(\S_{A}^{+})\to C(\S_{A}^{+})$, by
\[ L_{\varphi}f(x)=\sum_{\sigma y =x}e^{\varphi(y)}f(y). \]
If $\varphi$ is Bowen then there exists a unique measure $\nu$ such that $L_{\varphi}^{\ast}\nu = \rho(L_{\varphi}) \nu$ and a measurable function $h>0$ such that $L_{\varphi}h = \rho(L_{\varphi})h$, where $\rho(L_{\varphi})$ is the spectral radius of $L_{\varphi}$. The Gibbs measure is then constructed as $\mu_{\varphi}=hd\nu$. For background on Gibbs measures we refer the reader to Bowen's book \cite{bowen1975equilibrium} as well as \cite{MR1783787} for potentials less regular then H\"older. Let's consider an example of a hidden Markov measure.

\begin{example}\label{HMMexample}
	Let $\mu_{S}$ be the Markov measure on $\set{0,1,2}^{\N}$ defined by the stochastic matrix
	\[ S=\bmat{1/3&1/3&1/3 \\ 1/3 & 0 &2/3 \\ 1/6&1/6&2/3 }. \]
	Again suppose that the states $0$ and $1$ are labeled red and $2$ is labeled blue. That is $\pi:\set{0,1,2}^{\N}\to \set{r,b}^{\N}$ is the map induced by the function $\pi(0)=\pi(1)=r$ and $\pi(2)=b$. Let $\nu$ be the left eigenvector for $S$. Recall that the measure $\mu_{S}$ is the Gibbs state for the potential $\varphi(x)=\log S_{x_{0}x_{1}}$. The transfer operator for $\varphi$ preserves the subspace $\spn\set{\chi_{[a]}:a \in \S}$ and the matrix representation of $L_{\varphi}$ in the standard basis is given by the transpose of $S$ that is
	\[ L_{\varphi}=\bmat{1/3&1/3&1/6 \\ 1/3 & 0 &1/6 \\ 1/3&2/3&2/3 }. \]
	It is well known that there is a formula for the $\pi_{\ast}\mu_{S}$ measure of cylinder sets in terms of the matrices
	\begin{gather*}
	\L_{rr}= \bmat{1/3& 1/3&0 \\ 1/3 & 0 &0 \\ 0&0&0} \text{ , }\L_{rb}=\bmat{ 0& 0&1/3 \\ 0 & 0 &2/3 \\ 0 & 0 & 0 }\\
	\L_{br}= \bmat{ 0& 0&0 \\ 0 & 0 & 0 \\ 1/6&1/6&0 } \text{ , } \L_{bb}=\bmat{ 0& 0&0 \\ 0 & 0 &0 \\ 0& 0 &2/3 }.
	\end{gather*}
	That is we can write
	\[ \pi_{\ast}\mu_{S}[rrbr]=[1]\L_{rb}\L_{br}\L_{rr}\nu^{T}. \]
	where $[1]$ is the row vector of all $1$'s and $\nu$ is the left eigenvector for $S$ normalized so that $\sum_{i}\nu_{i}=1$. In this example all of the products of the matrices $\L_{ij}$ of length $2$ map one cone strictly inside another. They are thus strict contractions of the Hilbert projective metric and this implies $\pi_{\ast}\mu_{S}$ is the Gibbs state for a H\"older continuous potential, see \cite{Yoo2010} for details.
\end{example}

This example demonstrates two things. First the presence of excluded words (or more accurately the structure of the fibers $\pi^{-1}(y)$) has a significant impact on the regularity properties of the $g$ function for hidden Markov measures. Second it suggests a method for extending results beyond hidden Markov measures where the matrices $\L_{rr}$, $\L_{rb}$, $\L_{br}$ and $\L_{rr}$ are replaced by suitable operators on acting on subspaces of $C(\S_{A}^{+})$.

\section{The classical uniqueness regimes}\label{sec:CUR}

It is known from \cite{Yoo2010} that factors of Markov measures under a fiber-wise sub-positive mixing map have H\"older continuous $g$ functions. We have seen that factors of H\"older Gibbs states on full shifts have H\"older continuous $g$ functions. The goal of this section is to show that in fact all of the classical uniqueness regimes from definition \ref{defofCUR} are closed under fiber-wise sub-positive mixing factor maps. In particular in this section we will prove theorem \ref{mainthm}. Let $L_{\varphi}$ be the Ruelle operator associated to $\varphi$ and assume without loss that $\rho(L_{\varphi})=1$, in other words the pressure of $\varphi$  is $0$. Take $\nu$ such that $\L^{\ast}\nu=\nu$. Assume that $\varphi$ is Bowen.

Let $\pi:\S \to \ol{\S}$ be a map inducing a $1$-block factor map $\pi:\S_{A}^{+}\to Y \subseteq \ol{\S}^{\N}$ and assume that $\pi$ is fiber-wise sub-positive mixing. For each $i,j \in \S$ with $ji$ admissible in $\S_{A}^{+}$ define the operator
\[ L_{ij}f(x) = e^{\varphi(jx)}f(jx)\chi_{[i]}(x). \]
For example in the case of the Markov measure $\mu_{S}$ from example \ref{HMMexample} $L_{00}$ is simply the matrix
\[ \bmat{1/3 & 0  &0  \\ 0 & 0 & 0  \\ 0 & 0 &0 }. \]
As $\mu_{S}$ is the Gibbs state for the potential $\varphi(x) =\log S_{x_{0}x_{1}}$. For $bb'$ admissible in $Y$ define an operator $\L_{bb'}:C(\S_{A}^{+})\to C(\S_{A}^{+})$ by
\[ \L_{bb'}f:=\sum_{\pi(i)=b,\pi(j)=b'}L_{ij}f. \]
In the case of example \ref{HMMexample} the operator $\L_{br}=L_{20}+L_{21}$ is the matrix
\[ \L_{br}=\bmat{ 0 & 0 &0 \\ 0 & 0 & 0 \\ 1/3 & 2/3 & 0 }. \]
Given a word $w=w_{0}\cdots w_{n}$ admissible in $\S_{A}^{+}$ define
\[ L_{w}:= L_{w_{n}w_{n-1}}\cdots L_{w_{2}w_{1}}L_{w_{1}w_{0}} \]
and notice that
\begin{align*}
L_{w_{2}w_{1}}L_{w_{1}w_{0}}f(x)&=e^{\varphi(w_{0}x)}L_{w_{1}w_{0}}f(w_{1}x)\chi_{[w_{2}]}(x)\\
&=e^{\varphi(w_{0}x)}\left(e^{\varphi(w_{0}w_{1}x)}f(w_{0}w_{1}x)\chi_{[w_{1}]}(w_{1}x)\right)\chi_{[w_{2}]}(x)\\
&= e^{S_{2}\varphi(w_{0}w_{1}x)}f(w_{0}w_{1}x)\chi_{[w_{2}]}(x).
\end{align*}
Thus by iteration
\[ L_{w}f(x)=e^{S_{n}\varphi(w_{0}\cdots w_{n-1}x)}f(w_{0}\cdots w_{n-1}x)\chi_{[w_{n}]}(x). \]
Similarly for $\ol{w}=\ol{w}_{0}\cdots \ol{w}_{n}$ admissible in $Y$ define 
\[ \L_{\ol{w}} := \L_{\ol{w}_{n}\ol{w}_{n-1}}\cdots \L_{\ol{w}_{2}\ol{w}_{1}}\L_{\ol{w}_{1}\ol{w}_{0}}. \]

We will make use of a strategy of using sequences of cones to obtain sub exponential rates of convergence that has been used previously in the decay of correlations literature \cite{kondah1997vitesse}. Give a metric $d$ and a symbol $a\in \S$ define a cone
\[ \c([a],d)=\set{f \in C([a]):f\geq 0 \text{ and }f(x)\leq f(x')e^{d(x,x')}\text{ for all }x,x' \in [a]} \] 
and given a number $B>0$ define a cone
\[ \c([a],+B)=\set{f\in C([a]):f\geq 0\text{ and }f(x)\leq e^{B}f(x')\text{ for all }x,x'\in [a]}. \]
We now produce a sequence of metrics for which long products of the operators $L_{ij}$ map $\c([a],d_{j})$ to $\c([b],\sigma d_{j+1})$ for some $0<\sigma <1$. This is crucial to bounding the projective diameter (see corollary \ref{ThetaBoundRegularityCones}).

\begin{procedure}\label{coneconstructionprocedure}
	This procedure has two inputs. A function $\varphi$ with the Bowen property (which is fixed throughout this section) and a sequence $\alpha_{k}$ which is positive and decreasing to $0$ (we will make different choices for $\alpha_{k}$ depending on our needs). What it produces is a sequence of natural numbers $\set{n_{j}}_{j=0}^{\infty}$ and a sequence of metrics $\set{d_{j}}_{j=0}^{\infty}$ with specific properties. Construct a sequence $\set{n_{j}}_{j=0}^{\infty}$ of numbers and metrics $\set{d_{j}}_{j=0}^{\infty}$ in following way. Set
	\[ \alpha(0,k)=\alpha_{k} \text{ and }\alpha(m,k)=\var_{m+k}S_{m}\varphi. \]
	Fix some $0<\sigma <1$. Inductively define a sequence $\set{n_{i}}_{i=1}^{\infty}$: take $n_{1}$ such that for all $n \geq n_{1}$ we have $\alpha(0,n)\leq \frac{\sigma}{2}$. As $\alpha_{k}$ converges to $0$ and $S_{n_{1}}\varphi$ is continuous we may take $n_{2}$ such that for all $n \geq n_{2}$ we have that
	\[ \alpha(0, n_{1}+n)\leq \frac{\sigma^{2}}{2^{2}} \text{ and }\alpha(n_{1}, n)\leq \frac{\sigma }{2}. \]
	Continue in this way, that is given $\set{n_{i}}_{i=1}^{j-1}$ choose $n_{j}$ such that for all $n \geq n_{j}$ we have 
	\begin{gather*}
	\alpha(0, n_{1}+n_{2}+\cdots +n_{j-1}+n)\leq \frac{\sigma^{j}}{2^{j}},\\
	\alpha(n_{1}, n_{2}+n_{3}+\cdots +n_{j-1}+n)\leq \frac{\sigma^{j-1}}{2^{j-1}} ,\\
	\alpha(n_{2},n_{3}+\cdots +n_{j-1}+n)\leq \frac{\sigma^{j-2}}{2^{j-2}},\\
	\vdots \\
	\alpha(n_{i}, n_{i+1}+n_{i+2}+\cdots +n_{j-1}+n)\leq \frac{\sigma^{j-i}}{2^{j-i}}\\
	\vdots \\
	\alpha(n_{j-1},n)\leq \frac{\sigma}{2}
	\end{gather*}
	Define 
	\begin{equation}
	d_{j,k}=\sum_{i=0}^{j}\frac{\alpha(n_{i},n_{i+1}+n_{i+2}+\cdots +n_{j}+k)}{\sigma^{j-i+1}} 
	\end{equation}
	where $n_{0}:=0$. Notice that for each $j$, $d_{j,k}\xrightarrow{k \to \infty}0$ thus we can then define a sequence of metrics $\set{d_{j}}_{j=0}^{\infty}$ by
	\begin{equation}\label{eq:sequenceofmetric}
	d_{j}(x,y)=d_{j,k(x,y)} 
	\end{equation}
	where $k(x,y)=\min\set{i:x_{i}\neq y_{i}}$. Notice that $d_{0,k}=\alpha_{k}$.
\end{procedure}

\begin{remark}
	Unless $\varphi$ is H\"older $n_{j} \xrightarrow{j \to \infty}\infty$. In the case of a H\"older potential we can have that $n_{j}=1$ and $d_{j,k}=\frac{\abs{\varphi}_{\theta}\theta^{k+1}}{\sigma -\theta}$ for all $j$.
\end{remark}

What we have gained from this construction of $\set{n_{i}}_{i=0}^{\infty}$ and $d_{j,k}$ is the following. First, 
\begin{align*}
\alpha(n_{j+1},k)+d_{j,n_{j+1}+k}&=\alpha(n_{j+1},k)+\sum_{i=0}^{j}\frac{\alpha(n_{i},n_{i+1}+n_{i+2}+\cdots +n_{j}+n_{j+1}+k)}{\sigma^{j-i+1}}\\
&= \sigma \left(\frac{\alpha(n_{j+1},k)}{\sigma}+\sum_{i=0}^{j}\frac{\alpha(n_{i},n_{i+1}+n_{i+2}+\cdots +n_{j}+n_{j+1}+k)}{\sigma^{j-i+2}}\right)\\
&= \sigma \left(\sum_{i=0}^{j+1}\frac{\alpha(n_{i},n_{i+1}+n_{i+2}+\cdots +n_{j}+n_{j+1}+k)}{\sigma^{(j+1)-i+1}}\right)\\
&= \sigma d_{j+1,k}.
\end{align*}
This implies that $n_{j+1}$ products of the operators $L_{ij}$ map a cone of the type $\c([a],d_{j})$ into a cone of the type $\c([b],\sigma d_{j+1})$ which is the content of lemma \ref{IndividualConeMapping}. Second, we have that for $k \geq 0$
\[ d_{j,k}=\sum_{i=0}^{j}\frac{\alpha(n_{i},n_{i+1}+n_{i+2}+\cdots +n_{j-1}+n_{j}+k)}{\sigma^{j-i+1}}\leq \frac{1}{\sigma}\sum_{i=0}^{j-1}\frac{1}{2^{j-i}}+\frac{\alpha(n_{j},k)}{\sigma}\leq \frac{2+K}{\sigma}. \]
Where $K$ is the constant from the Bowen property for $\varphi$, this implies that for all $j$ and $a$ we have that $\c([a],d_{j})\subseteq \c([a],+B)$ where $B=\frac{2+K}{\sigma}$.

\begin{lemma}\label{IndividualConeMapping}
	For any $j$ and word $w=w_{0}\cdots w_{n_{j+1}}$ we have that 
	\[ L_{w}:\c([w_{0}],d_{j})\to \c([w_{n_{j+1}}], \sigma d_{j+1}).\]
\end{lemma}
\begin{proof}
	Suppose that $f \in \c([w_{0}],d_{j})$. For $x,x' \in [w_{n_{j+1}}]$ with $x_{i}=x_{i}'$ for all $0\leq i \leq k-1$. We have by definition that 
	\[ \abs{S_{n_{j+1}}\varphi(w_{0}\cdots w_{n_{j+1}-1}x)-S_{n_{j+1}}\varphi(w_{0}\cdots w_{n_{j+1}-1}x')}\leq \alpha(n_{j+1},k)  \]
	and
	\begin{align*}
	\frac{f(w_{0}\cdots w_{n_{j+1}-1}x)}{f(w_{0}\cdots w_{n_{j+1}-1}x')}&\leq \exp \left[d_{j}(w_{0}\cdots w_{n_{j+1}-1}x,w_{0}\cdots w_{n_{j+1}-1}x')\right]\\
	&=\exp \left[d_{j,n_{j+1}+k} \right] 
	\end{align*}
	because $f \in \c([w_{n_{j+1}}],d_{j})$ and $w_{0}\cdots w_{n_{j+1}-1}x, w_{0}\cdots w_{n_{j+1}-1}x'$ agree for $n_{j+1}+k$ places. Therefore
	\begin{align*}
	&L_{w}f(x)\\
	&=e^{S_{n_{j+1}}\varphi(w_{0}\cdots w_{n_{j+1}-1}x)}f(w_{0}\cdots w_{n_{j+1}-1}x)\\
	&=e^{S_{n_{j+1}}\varphi(w_{0}\cdots w_{n_{j+1}-1}x)-S_{n_{j+1}}\varphi(w_{0}\cdots w_{n_{j+1}-1}x')+S_{n_{j+1}}\varphi(w_{0}\cdots w_{n_{j+1}-1}x')}\\
	&\indent \times \frac{f(w_{0}\cdots w_{n_{j+1}-1}x)}{f(w_{0}\cdots w_{n_{j+1}-1}x')}f(w_{0}\cdots w_{n_{j+1}-1}x')\\
	&\leq \exp \left[ \alpha(n_{j+1},k)+d_{j,n_{j+1}+k} \right]L_{w}f(x')\\
	&=e^{\sigma d_{j+1}(x,x')}L_{w}f(x')
	\end{align*}
	Hence $L_{w}f \in \c([w_{n_{j+1}}], \sigma d_{j+1})$.
\end{proof}

The strategy of the proof will be the following. Define $\psi:Y \to \R$ by
\begin{equation}\label{PsiDef}
\psi(y)=\lim_{m \to \infty} \log \frac{ \inn{\L_{y_{m}y_{m-1}}\cdots \L_{y_{2}y_{1}}\L_{y_{1}y_{0}}1, \nu}}{\inn{\L_{y_{m}y_{m-1}}\cdots \L_{y_{2}y_{1}}1, \nu}}. 
\end{equation}
Recall that if $\psi$ is Bowen then it has a unique Gibbs state $\mu_{\psi}$ which satisfies the Gibbs inequality 
\[ C^{-1}\leq \frac{\mu_{\psi}[y_{0}\cdots y_{n-1}]}{e^{S_{n}\psi(y)}}\leq C \]
(that the pressure is $0$ is a consequence of the assumption that $\rho(L_{\varphi})=1$). A computation (Proposition \ref{BowenProp}) shows that 
\[ e^{S_{n}\psi(y)}\approx \pi_{\ast}\nu[y_{0}\cdots y_{n-1}]\approx \pi_{\ast}\mu_{\varphi}[y_{0}\cdots y_{n-1}]. \]
Thus $\mu_{\psi}$ and $\pi_{\ast}\mu_{\varphi}$ are mutually absolutely continuous and therefore equal as they are ergodic. $\pi_{\ast}\mu_{\varphi}$ is then the Gibbs state for $\psi$ and theorem \ref{mainthm} will follow by proving that $\psi$ is H\"older, Walters, or Bowen respectively.

This proof strategy is similar to \cite{kempton2011factors} and \cite{pollicott2011factors} with the observation that the choice of the measure $\nu$ is arbitrary (although it also allows us to obtain corollary \ref{maincor}). The notable difference being that we appeal to cone-theoretic techniques while \cite{kempton2011factors} and \cite{pollicott2011factors} are based on ``hands on'' bounds. That is, the function defined in \cite{kempton2011factors} and \cite{pollicott2011factors} is essentially the same as \eqref{PsiDef} with $\nu$ replaced by a point mass. The real benefit to our method is that we can prove that classes are closed under factor maps. This is in contrast to \cite{kempton2011factors} and \cite{pollicott2011factors} where the main application is continuity rates.

We must now construct a sequence of cones. To do so we identify $C(\S_{\A})$ with $\bigoplus_{a \in \S}C([a])$ and define cones as in the proposition \ref{directsumofcones}. Given a word $\ol{w}=\ol{w}_{0}\cdots \ol{w}_{n}$ admissible in $Y$ and a metric $d$, define cones 
\[ 
\c(\text{in}, \ol{w}, d)= \bigoplus_{a:\exists w, \pi(w)=\ol{w}, w_{0}=a}\c([a],d) \]
and
\[ \c(\text{out}, \ol{w}, d)=\bigoplus_{a:\exists w, \pi(w)=\ol{w}, w_{n}=a}\c([a],d). \]
We will give names to the sets 
\[ \ol{w}_{\text{in}}= \set{a:\exists w, \pi(w)=\ol{w}, w_{0}=a}\subseteq \S \]
and 
\[ \ol{w}_{\text{out}}=\set{a:\exists w, \pi(w)=\ol{w}, w_{n}=a}\subseteq \S.\]
In addition we will call
\[ \c(\text{in},\ol{w},+)&=\bigoplus_{a \in \ol{w}_{\text{in}}}C([a],+)\\
&= \set{f \geq 0: f(x)=0 \text{ for all }x \notin \bigcup_{a \in \ol{w}_{\text{in}}}[a]} \]
and
\[ \c(\text{in},\ol{w},+B)&=\bigoplus_{a \in \ol{w}_{\text{in}}}C([a],+B)\\
&= \set{f \in \c(\text{in},\ol{w},+) : f(x)\leq e^{B}f(x') \text{ for all }x,x' \text{ with }x_{0}=x_{0}'} \]
and similarly for ``out''. To get a sense of these definitions observe that in example \ref{HMMexample} we have that
\[ \c(\text{in}, br , +)=\set{\bmat{a_{0} \\ a_{1}\\ 0}:a_{0},a_{1}\geq 0} \]
and 
\[ \c(\text{out}, br , +)=\set{\bmat{0 \\ 0 \\ a_{2}}:a_{2}\geq 0}. \]
Recall that because the potential in example \ref{HMMexample} depends only on $2$ coordinates we can think of these cones as the non-negative quadrant of $\spn\set{\chi_{[a]}: a \in \S}\cong \R^{\abs{\S}}$. In general the cones $\c(\text{in}, \ol{w}, d), \c(\text{out}, \ol{w}, d)$ are simply non-negative functions supported on the cylinder sets corresponding to $\ol{w}_{\text{in}}$ and $\ol{w}_{\text{out}}$ for which the restriction to any cylinder set $[a]$ is in the cone $\c([a],d)$.

\begin{lemma}
	Suppose that $\ol{w}=\ol{w}_{0}\cdots \ol{w}_{n_{j+1}}$ is a word admissible in $Y$. We have that 
	\[\L_{\ol{w}}:\c(\text{in},\ol{w},d_{j})\to \c(\text{out},\ol{w}, \sigma d_{j+1}).\]
\end{lemma}
\begin{proof}
	This is a consequence of the fact that $\L_{\ol{w}}=\sum_{\pi(w)=\ol{w}}L_{w}$ and lemma \ref{IndividualConeMapping}.
\end{proof}

\begin{lemma}\label{M}
	Let $N$ be as in definition \ref{fibermixingdef}. There exists a constant $C$ such that 
	\[ \Theta_{\c(\text{out}, \ol{w},+)}(\L_{\ol{w}}f,1)\leq C\]
	for all $f\in \c(\text{in}, \ol{w}, +B)$ and admissible words $\ol{w}=\ol{w}_{0}\cdots \ol{w}_{N}$.
\end{lemma}
\begin{proof}
	Let $z \in \bigcup_{\pi(a)=w_{0}}[a]$ be such that $f(z)=\norm{f}_{\infty}$. By the assumption that $\pi$ is fiberwise sub-positive mixing we have that for any $x \in \bigcup_{\pi(a)=\ol{w}_{N}}[a]$ there exists a word $w$ with $\pi(w)=\ol{w}$ and $w_{0}=z_{0}, w_{N}=x_{0}$. Thus
	\begin{align*}
	\L_{\ol{w}}f(x)\geq L_{w}f(x)= e^{S_{N}\varphi(w_{0}\cdots w_{N-1}x_{0}\cdots)}f(w_{0}\cdots w_{N-1}x_{0}\cdots)\geq e^{-N\norm{\varphi}_{\infty}}\norm{f}_{\infty}e^{-B}
	\end{align*}
	on the other hand we have that
	\[ \L_{\ol{w}}f(x)\leq \norm{\L_{\ol{w}}}_{\text{op}}\norm{f}_{\infty}\leq \norm{L_{\varphi}}^{N}\norm{f}_{\infty}. \]
	Therefore
	\[ \Theta_{\c(\text{out}, \ol{w}, +)}(\L_{\ol{w}}f, 1)\leq \log\left(\frac{\norm{L_{\varphi}}_{\text{op}}^{N} }{e^{-B-N\norm{\varphi}_{\infty}}}\right). \]
\end{proof}

\begin{remark}
	Note by potentially taking $n_{j}$ slightly larger if necessary we can have that $n_{j}\geq N$ for all $j$.
\end{remark}

\begin{lemma}\label{ContractionLemma}
	\begin{enumerate}
		\item 
		There exists a constant $D$ such that for any $j$ and $\ol{w}=\ol{w}_{0}\cdots \ol{w}_{n_{j+1}}$ admissible in $Y$ we have that 
		\[ \diam_{\c(\text{out},\ol{w},d_{j+1})}(\L_{\ol{w}}\c(\text{in},\ol{w},d_{j}))\leq D. \]
		
		\item 
		Let $j \geq 0$, $k\geq 1$ and set $m=\sum_{i=1}^{k}n_{j+i}$. For any word $y_{0} \cdots y_{m}$ admissible in $Y$ we have that 
		\[ \Theta_{\c(\text{out}, y_{n_{j+1}}\cdots y_{m}, d_{j+2})}(\L_{y_{m}y_{m-1}}\cdots \L_{y_{1}y_{0}}f,\L_{y_{m}y_{m-1}}\cdots \L_{y_{1}y_{0}}g)\leq \gamma^{k}\Theta_{\c(\text{in},y_{0}\cdots y_{n_{j+1}},d_{j})}(f,g) \]
		where $\gamma:= \tanh(D/4)$.
	\end{enumerate}
\end{lemma}
\begin{proof}
	\begin{enumerate}
		\item 
		Suppose that $f,g \in \c(\text{in},\ol{w},d_{j})$ as $\L_{\ol{w}}f,\L_{\ol{w}}g \in \c(\text{out}, \ol{w}, \sigma d_{j+1})$ we have by corollary \ref{ThetaBoundRegularityCones} that 
		\begin{align*}
		\Theta_{\c(\text{out}, \ol{w}, d_{j+1})}(\L_{\ol{w}}f,\L_{\ol{w}}g) &\leq 2 \log \frac{1+\sigma}{1-\sigma}+\Theta_{\c(\text{out}, \ol{w}, +)}(\L_{\ol{w}}f,\L_{\ol{w}}g)\\
		&\leq 2 \log \frac{1+\sigma}{1-\sigma}+2C 
		\end{align*}
		where $C$ is as in lemma \ref{M}.
		
		\item 
		We will show the second assertion when $k=2$ the general case for $k \geq 1$ is similar. Set $\ol{w}^{1}=y_{0}\cdots y_{n_{j+1}}$ and $\ol{w}^{2}=y_{n_{j+1}}\cdots y_{m}$. Notice that
		\[ \L_{y_{m}y_{m-1}}\cdots \L_{y_{1}y_{0}}= \L_{\ol{w}^{2}}\L_{\ol{w}^{1}}. \]
		Define a map $R_{\ol{w}^{2}_{\text{in}}}:\c(\text{out}, \ol{w}^{1}, d_{j+1}) \to \c(\text{in}, \ol{w}^{2},  d_{j+1})$ by 
		\[R_{\ol{w}^{2}_{\text{in}}}f(x)=\sum_{a \in \ol{w}^{2}_{\text{in}}}\chi_{[a]}(x)f(x).\]
		Notice that for any $f \in C(\S_{A})$, $f-R_{\ol{w}^{2}_{\text{in}}}f\in \ker \L_{\ol{w}^{2}}$ because is it supported on cylinder sets for which no word projecting to $\ol{w}^{2}$ begins ($R_{\ol{w}^{2}_{\text{in}}}$ is nothing but a restriction map). Thus $\L_{\ol{w}^{2}}f = \L_{\ol{w}^{2}}R_{\ol{w}^{2}_{\text{in}}}f$. To get a sense of what is happening consider the following diagram
		\begin{align}\label{restrictionconecontraction}
		\c(\text{in}, \ol{w}^{1}, d_{j}) \xrightarrow{\L_{\ol{w}^{1}}}\c(\text{out}, \ol{w}^{1}, d_{j+1}) \xrightarrow{R_{\ol{w}^{2}_{\text{in}}}}\c(\text{in}, \ol{w}^{2},  d_{j+1})\xrightarrow{\L_{\ol{w}^{2}}}\c(\text{out}, \ol{w}^{2},  d_{j+2})
		\end{align}
		Hence
		\begin{align*}
		\Theta_{\c(\text{out}, \ol{w}^{2}, d_{j+2})}(\L_{\ol{w}^{2}}\L_{\ol{w}^{1}}f,\L_{\ol{w}^{2}}\L_{\ol{w}^{1}}g)&=\Theta_{\c(\text{out}, \ol{w}^{2}, d_{j+2})}(\L_{\ol{w}^{2}}R_{\ol{w}^{2}_{\text{in}}}\L_{\ol{w}^{1}}f,\L_{\ol{w}^{2}}R_{\ol{w}^{2}_{\text{in}}}\L_{\ol{w}^{1}}g)\\
		&\leq \gamma \Theta_{\c(\text{in}, \ol{w}^{2}, d_{j+1})}(R_{\ol{w}^{2}_{\text{in}}}\L_{\ol{w}^{1}}f,R_{\ol{w}^{2}_{\text{in}}}\L_{\ol{w}^{1}}g)\\
		&\leq \gamma \Theta_{\c(\text{out}, \ol{w}^{1}, d_{j+1})}(\L_{\ol{w}^{1}}f,\L_{\ol{w}^{1}}g)\\
		&\leq \gamma^{2}\Theta_{\c(\text{in},\ol{w}^{1},d_{j})}(f,g) 
		\end{align*}
		by theorem \ref{BirkhoffContraction}.
	\end{enumerate}
\end{proof}

\begin{proposition}\label{BowenProp}
	Suppose that $\varphi$ is Bowen.
	\begin{enumerate}
		\item 
		The sequence of functions $\set{f_{m}}\subseteq C(Y)$ defined by
		\[ f_{m}(y):=\log \frac{ \inn{\L_{y_{m}y_{m-1}}\cdots \L_{y_{2}y_{1}}\L_{y_{1}y_{0}}1, \nu}}{\inn{\L_{y_{m}y_{m-1}}\cdots \L_{y_{2}y_{1}}1, \nu}}= \log \frac{\pi_{\ast}\nu [y_{0}y_{1}\cdots y_{m}]}{\pi_{\ast}\nu [y_{1}\cdots y_{m}]} \]
		converges uniformly to a function $\psi:Y \to \R$. Moreover $\psi$ is Bowen.
		
		\item 
		If $\mu_{\psi}$ is the unique Gibbs states for $\psi$ then there exist constants $C_{1},C_{2}>0$ such that
		\[ C_{1}\leq \frac{\mu_{\psi}[y_{0}\cdots y_{n-1}]}{\pi_{\ast}\mu_{\varphi}[y_{0}\cdots y_{n-1}]}\leq C_{2}. \]
		for all $y$.
		
		\item 
		$\mu_{\psi}=\pi_{\ast}\mu_{\varphi}$.
	\end{enumerate}
\end{proposition}
\begin{proof}
	\begin{enumerate}
		\item 
		Let $\e>0$. Let $\alpha_{k}=\var_{k}\varphi$, set $d_{j}$ to be the metrics as in equation \eqref{eq:sequenceofmetric} and notice that this implies that $L_{ij}1 \in \c([i],d_{0})$. Let $D$ and $\gamma$ be as in lemma \ref{ContractionLemma}. Take $N=1+\sum_{i=1}^{k}n_{i}$ where $k$ is chosen such that $\gamma^{k-1}D<\e$. Suppose that $y \in Y$ and that $n,m \geq N$. Then
		\begin{align*}
		&\abs{f_{n}(y)-f_{m}(y)}\\
		&=\abs{\log \frac{ \inn{\L_{y_{n}y_{n-1}}\cdots \L_{y_{2}y_{1}}\L_{y_{1}y_{0}}1, \nu}}{\inn{\L_{y_{n}y_{n-1}}\cdots \L_{y_{2}y_{1}}1, \nu}}-\log \frac{ \inn{\L_{y_{m}y_{m-1}}\cdots \L_{y_{2}y_{1}}\L_{y_{1}y_{0}}1, \nu}}{\inn{\L_{y_{m}y_{m-1}}\cdots \L_{y_{2}y_{1}}1, \nu}}}\\
		&=\abs{\log \frac{ \inn{\L_{y_{N}y_{N-1}}\cdots \L_{y_{2}y_{1}}\L_{y_{1}y_{0}}1, \nu_{y,n}}\inn{\L_{y_{N}y_{N-1}}\cdots \L_{y_{2}y_{1}}1, \nu_{y,m}}}{\inn{\L_{y_{N}y_{N-1}}\cdots \L_{y_{2}y_{1}}1, \nu_{y,n}}\inn{\L_{y_{N}y_{N-1}}\cdots \L_{y_{2}y_{1}}\L_{y_{1}y_{0}}1, \nu_{y,m}}}}
		\end{align*}
		where
		\[ \nu_{y,n}= \L_{y_{N+1}y_{N}}^{\ast}\cdots \L_{y_{n}y_{n-1}}^{\ast}\nu \]
		and similarly for $\nu_{y,m}$. For the sake of notation divide $y_{1}y_{2}\cdots y_{N}$ into words $\ol{w}^{1}=y_{1}\cdots y_{n_{1}+1}$, $\ol{w}^{2}=y_{n_{1}+1}\cdots y_{n_{1}+n_{2}+1}$ and so on. Thus by lemma \ref{DualconeMetricComputation} 
		\[ \abs{f_{n}(y)-f_{m}(y)}\leq \Theta_{+}(\L_{y_{n}y_{n-1}}\cdots \L_{y_{2}y_{1}}(\L_{y_{1}y_{0}}1), \L_{y_{n}y_{n-1}}\cdots \L_{y_{2}y_{1}}1), \]
		where $\Theta_{+}$ is the Hilbert metric on the the cone of non-negative functions. In addition by our choice of $\alpha_{k}$ we have
		\[ R_{\ol{w}^{1}_{\text{in}}}1,R_{\ol{w}^{1}_{\text{in}}}\L_{y_{0}y_{1}}1 \in \c(\text{in}, \ol{w}^{1},d_{0}).\]
		Thus by theorem \ref{BirkhoffContraction} and lemma \ref{ContractionLemma} we have that
		\begin{align*}
		\abs{f_{n}(y)-f_{m}(y)}&\leq \Theta_{+}(\L_{y_{n}y_{n-1}}\cdots \L_{y_{2}y_{1}}(\L_{y_{1}y_{0}}1), \L_{y_{n}y_{n-1}}\cdots \L_{y_{2}y_{1}}1)\\
		&\leq \Theta_{\c(\text{out}, \ol{w}^{k}, d_{k})}(\L_{\ol{w}^{k}}\cdots \L_{\ol{w}^{1}}(\L_{y_{1}y_{0}}1), \L_{\ol{w}^{k}}\cdots \L_{\ol{w}^{1}}1)\\
		&\leq \gamma^{k-1}\Theta_{\c(\text{in},\ol{w}^{2},d_{1})}(R_{\ol{w}^{2}_{\text{in}}}\L_{\ol{w}^{1}}(\L_{y_{1}y_{0}}1),R_{\ol{w}^{2}_{\text{in}}}\L_{\ol{w}^{1}}1)\\
		&\leq\gamma^{k-1}\Theta_{\c(\text{out},\ol{w}^{1},d_{1})}(\L_{\ol{w}^{1}}(\L_{y_{1}y_{0}}1),\L_{\ol{w}^{1}}1)\leq \gamma^{k-1}D < \e.
		\end{align*}
		Therefore $\set{f_{n}}$ is Cauchy in $C(Y)$ and thus converges by completeness. To see that $\psi$ is Bowen notice that 
		\begin{align*}
		S_{n}\psi(y)&=\sum_{k=0}^{n-1}\lim_{m \to \infty}\log \frac{ \inn{\L_{y_{m}y_{m-1}}\cdots \L_{y_{k+2}y_{k+1}}\L_{y_{k+1}y_{k}}1, \nu}}{\inn{\L_{y_{m}y_{m-1}}\cdots \L_{y_{k+2}y_{k+1}}1, \nu}}\\
		&=\lim_{m \to \infty}\log \frac{\inn{\L_{y_{m}y_{m-1}}\cdots \L_{y_{n+1}y_{n}}\L_{y_{n}y_{n-1}}\cdots \L_{y_{1}y_{0}}1, \nu}}{\inn{\L_{y_{m}y_{m-1}}\cdots \L_{y_{n+1}y_{n}}1, \nu}}.
		\end{align*}
		Thus for any $n$ and $y,y' \in Y$ such that $y_{i}= y_{i}'$ for all $0 \leq i \leq n-1$ we have that
		\begin{align*}
		&\abs{S_{n}\psi(y)- S_{n}\psi(y')}\\
		&= \lim_{m \to \infty}\abs{\log \frac{\inn{\L_{y_{m}y_{m-1}}\cdots \L_{y_{n+1}y_{n}}\L_{y_{n}y_{n-1}}\cdots \L_{y_{1}y_{0}}1, \nu}\inn{\L_{y_{m}'y_{m-1}'}\cdots \L_{y_{n+1}'y_{n}'}1, \nu}}{\inn{\L_{y_{m}y_{m-1}}\cdots \L_{y_{n+1}y_{n}}1,\nu}\inn{\L_{y_{m}'y_{m-1}'}\cdots \L_{y_{n+1}'y_{n}'}\L_{y_{n}'y_{n-1}'}\cdots \L_{y_{1}'y_{0}'}1, \nu}}}\\
		&= \lim_{m \to \infty}\abs{\log \frac{\inn{\L_{y_{n}y_{n-1}}\cdots \L_{y_{1}y_{0}}1, \nu_{y,m}}\inn{1, \nu_{y',m}}}{\inn{1, \nu_{y,m}}\inn{\L_{y_{n}y_{n-1}}\cdots \L_{y_{1}y_{0}}1, \nu_{y',m}}}}\\
		&\leq \Theta_{\c(\text{out}, \ol{w},+)}(\L_{y_{n}y_{n-1}}\cdots \L_{y_{1}y_{0}}1, 1).
		\end{align*}
		Notice that for any word $w=w_{0}\cdots w_{n}$ admissible in $\S_{A}^{+}$ and $x,x' \in [w_{n}]$ we have that
		\begin{align*}
		L_{w}1(x)&=e^{S_{n}\varphi(w_{0}\cdots w_{n-1}x)-S_{n}\varphi(w_{0}\cdots w_{n-1}x')+S_{n}\varphi(w_{0}\cdots w_{n-1}x')}\\
		&\leq e^{K} L_{w}1(x').
		\end{align*}
		Where $K$ is the constant in the Bowen property for $\varphi$. Thus for any word $\ol{w}$ admissible in $Y$ we have that $\L_{\ol{w}}1 \in \c(\text{out},\ol{w}, +B)$ (note that $K < \frac{2+K}{\sigma}=B$). Therefore restricting $\L_{\ol{w}}1$ as necessary (in the same way as the paragraph preceding equation \eqref{restrictionconecontraction}) we have that for $n \geq N+1$
		\begin{align*}
		\Theta_{\c(\text{out}, \ol{w},+)}(\L_{y_{n-1}y_{n-2}}\cdots\L_{y_{n-N}y_{n-N-1}}(\L_{y_{n-N-1}y_{n-N-2}}\cdots \L_{y_{1}y_{0}}1), 1)\leq C 
		\end{align*}
		where $N$ and $C$ are as in lemma \ref{M}. Hence $\psi$ is Bowen.
		
		\item 
		Notice that for any probability measure $\eta$ and cylinder set $[I]$ we have that
		\begin{align*}
		e^{-K}C_{1}\mu_{\varphi}([I])\leq e^{-K} \sup_{z:Iz \in \S_{A}}e^{S_{n}\varphi(Iz)}\leq \int e^{S_{n}\varphi(Iz)}d\eta(z) \leq \sup_{z:Iz \in \S_{A}}e^{S_{n}\varphi(Iz)}\leq C_{2}\mu_{\varphi}([I]).
		\end{align*}
		Thus
		\begin{align*}
		e^{S_{n}\psi(y)}&= \lim_{m \to \infty}\frac{\inn{\L_{y_{m}y_{m-1}}\cdots \L_{y_{n+1}y_{n}}\L_{y_{n}y_{n-1}}\cdots \L_{y_{1}y_{0}}1, \nu}}{\inn{\L_{y_{m}y_{m-1}}\cdots \L_{y_{n+1}y_{n}}1, \nu}}\\
		&= \lim_{m \to \infty}\inn{\L_{y_{n}y_{n-1}}\cdots \L_{y_{1}y_{0}}1, \frac{\nu_{y,m}}{\inn{1,\nu_{y,m}}}}\\
		&= \lim_{m \to \infty}\sum_{I:\pi(I)=y_{0}\cdots y_{n-1}}\int_{\S_{A}}e^{S_{n}\varphi(Iz)}d\left(\frac{\nu_{y,m}}{\inn{1,\nu_{y,m}}}\right)\\
		&\geq e^{-K}C_{1} \pi_{\ast}\mu_{\varphi}[y_{0}\cdots y_{n-1}]
		\end{align*}
		where 
		\[ \nu_{y,m}= \L_{y_{n+1}y_{n}}^{\ast}\cdots \L_{y_{m}y_{m-1}}^{\ast}\nu. \]
		Similarly
		\begin{align*}
		e^{S_{n}\psi(y)}\leq C_{2} \pi_{\ast}\mu_{\varphi}[y_{0}\cdots y_{n-1}].
		\end{align*}
		Observe that this implies that the pressure of $\psi$ is $0$. The result is then simply an application of the Gibbs inequality for $\mu_{\psi}$.
		
		\item 
		By $2$ be have that $\mu_{\psi}$ and $\pi_{\ast}\mu$ are mutually absolutely continuous. As they are ergodic this implies they are equal. 
	\end{enumerate}
\end{proof}

\begin{remark}
	Observe that if $\varphi$ is the logarithm of a $g$-function then $\nu = \mu_{\varphi}$ and the function $\psi$ is the logarithm of the $g$ function for $\pi_{\ast}\mu_{\varphi}$. Thus the proof of theorem \ref{mainthm} will also yield corollary \ref{maincor}.
\end{remark}

What remains to prove theorem \ref{mainthm} is the following.

\begin{theorem}
	Suppose that $\varphi$ is Walters (respectively H\"older). Then the function
	\[ \psi(y)=\lim_{m \to \infty}\log \frac{ \inn{\L_{y_{m}y_{m-1}}\cdots \L_{y_{2}y_{1}}\L_{y_{1}y_{0}}1, \nu}}{\inn{\L_{y_{m}y_{m-1}}\cdots \L_{y_{2}y_{1}}1, \nu}} \]
	is Walters (respectively H\"older).
\end{theorem}
\begin{proof}
	We will prove the result in the case that $\varphi$ is Walters. The H\"older case is similar to the previous section \cite{HolderGibbsStates}. Let $\alpha_{k}=\sup_{n \geq 1}\var_{n+k}S_{n}\varphi$ which by the assumption that $\varphi$ is Walters converges to $0$ (and thus can be used as an input for procedure \ref{coneconstructionprocedure}). Define the sequence $\set{n_{j}}$ based on procedure \ref{coneconstructionprocedure}. Set $d_{j}$ to be the sequences of metrics based on $\set{\alpha_{i}}$ and $\set{n_{j}}$ as in \eqref{eq:sequenceofmetric}.
	
	Let $\e>0$. To show that $\psi$ has the Walters property we need to show that there exists $N$ such that for all $j \geq N$, $\var_{n+j}S_{n}\psi < \e$ for all $n \geq 1$. To do so suppose that $n \geq 1$, $j \geq 0$ and that $y, y' \in Y$ are such that $y_{i}=y_{i}'$ for all $0 \leq i \leq n+j-1$. In the same way as the proof of proposition \ref{BowenProp} we have that
	\begin{align*}
	&\abs{S_{n}\psi(y)-S_{n}\psi(y')}\\
	&=\lim_{m \to \infty}\abs{\log \frac{ \inn{\L_{y_{n+j-1}y_{n+j-2}}\cdots \L_{y_{n+1}y_{n}}\L_{y_{n}y_{n-1}}\cdots \L_{y_{1}y_{0}}1, \nu_{y,m}}\inn{\L_{y_{n+j-1}y_{n+j-2}}\cdots \L_{y_{n+1}y_{n}}1, \nu_{y',m}}}{\inn{\L_{y_{n+j-1}y_{n+j-2}}\cdots \L_{y_{n+1}y_{n}}1, \nu_{y,m}}\inn{\L_{y_{n+j-1}y_{n+j-2}}\cdots \L_{y_{n+1}y_{n}}\L_{y_{n}y_{n-1}}\cdots \L_{y_{1}y_{0}}1, \nu_{y',m}}}}
	\end{align*}
	where
	\[ \nu_{y,m}= \L_{y_{n+j}y_{n+j-1}}^{\ast}\cdots \L_{y_{m}y_{m-1}}^{\ast}\nu \]
	and similarly for $\nu_{y',m}$. Thus by lemma \ref{DualconeMetricComputation}
	\[ \abs{S_{n}\psi(y)-S_{n}\psi(y')}\leq \Theta_{+}(\L_{y_{n+j-1}y_{n+j-2}}\cdots \L_{y_{n+1}y_{n}}(\L_{y_{n}y_{n-1}}\cdots \L_{y_{1}y_{0}}1), \L_{y_{n+j-1}y_{n+j-2}}\cdots \L_{y_{n+1}y_{n}}1). \]
	Let $D$ and $\gamma$ be as in lemma \ref{ContractionLemma}, take $k$ such that $\gamma^{k-1}D<\e$ and set $N=1+\sum_{i=1}^{k}n_{i}$ (note $N$ depends only on $\e$ and not on $n,y$ or $y'$). Assume that $j \geq N$. For notational purposes divide $y_{n}y_{n+1}\cdots y_{n+N}$ into words $\ol{w}^{1}=y_{n}\cdots y_{n+n_{1}}$, $\ol{w}^{2}=y_{n+n_{1}}\cdots y_{n+n_{1}+n_{2}}$ and so on. Notice that by our choice of the sequence $\set{\alpha_{k}}$ we have that for any $x \in \S_{A}^{+}$ 
	\[L_{x_{n}x_{n-1}}\cdots L_{x_{1}x_{0}}1 \in \c([x_{n}],d_{0}) \]
	and thus
	\[ R_{\ol{w}^{1}_{\text{in}}}1,R_{\ol{w}^{1}_{\text{in}}}\L_{y_{n}y_{n-1}}\cdots \L_{y_{1}y_{0}}1 \in \c(\text{in}, \ol{w}^{1},d_{0}). \]
	The picture to have in mind through the remainder of the proof is the following
	\begin{align*}
	\cdots\xleftarrow{R_{\ol{w}^{3}_{\text{in}}}}\c(\text{out}, \ol{w}^{2},  d_{2})\xleftarrow{\L_{\ol{w}^{2}}}\c(\text{in}, \ol{w}^{2},  d_{1})\xleftarrow{R_{\ol{w}^{2}_{\text{in}}}}\c(\text{out}, \ol{w}^{1}, d_{1})\xleftarrow{\L_{\ol{w}^{1}}}\c(\text{in}, \ol{w}^{1}, d_{0}).
	\end{align*}
	By lemma \ref{ContractionLemma} each $L_{\ol{w}^{i}}$ contracts the relevant cones by a factor of $\gamma$ and thus we have
	\begin{align*}
	\abs{S_{n}\psi(y)-S_{n}\psi(y')}&\leq \Theta_{+}(\L_{y_{n+j-1}y_{n+j-2}}\cdots \L_{y_{n+1}y_{n}}(\L_{y_{n}y_{n-1}}\cdots \L_{y_{1}y_{0}}1), \L_{y_{n+j-1}y_{n+j-2}}\cdots \L_{y_{n+1}y_{n}}1)\\
	&\leq \Theta_{\c(\text{out}, \ol{w}^{k}, d_{k})}(\L_{\ol{w}^{k}}\cdots \L_{\ol{w}^{1}}(\L_{y_{n}y_{n-1}}\cdots \L_{y_{1}y_{0}}1), \L_{\ol{w}^{k}}\cdots \L_{\ol{w}^{1}}1)\\
	&\leq \gamma^{k-1}\Theta_{\c(\text{in},\ol{w}^{2},d_{1})}(R_{\ol{w}^{2}_{\text{in}}}\L_{\ol{w}^{1}}(\L_{y_{n}y_{n-1}}\cdots \L_{y_{1}y_{0}}1),R_{\ol{w}^{2}_{\text{in}}}\L_{\ol{w}^{1}}1)\\
	&\leq\gamma^{k-1}\Theta_{\c(\text{out},\ol{w}^{1},d_{1})}(\L_{\ol{w}^{1}}(\L_{y_{n}y_{n-1}}\cdots \L_{y_{1}y_{0}}1),\L_{\ol{w}^{1}}1)\leq \gamma^{k-1}D < \e.
	\end{align*}
	Thus taking the supremum over $y,y'$ we have that $\var_{n+j}S_{n}\psi<\e$ and taking the supremum over $n$ gives $\sup_{n \geq 1}\var_{n+j}S_{n}\psi <\e$. Therefore $\sup_{n}\var_{n+j}S_{n}\psi \xrightarrow{ j \to \infty}0$ and $\psi$ is Walters.
\end{proof}

\section{Final remarks}
\begin{enumerate}
	\item 
	Similar results hold on countable state topological Markov shifts with a strong ``mixing in fibers'' assumption on the factor map.
	
	\item 
	The theory of factors of Gibbs measures in dimension greater then $1$ is well known to be filled with strange phenomena such as loss of Gibbsianness. This is in stark contrast to the results we have presented in this chapter. It is worth pointing out that the uniqueness regimes we have considered in this article do not exhibit many pathologies which are prevalent in higher dimensions. In particular they do not undergo phase transitions.
	
	\item 
	One interesting direction to move in an attempt to bridge the gap between our results and those in higher dimension would be to consider factors of Markov measures on countable state topological Markov shifts. These represent an intriguing class with respect to the current theory. On the one hand they can exhibit phase transitions and on the other they are fairly tractable objects with a well developed theory. It seems likely that one could understand in a fairly explicit way the existence of so called ``hidden phase transitions'' in these models and their connection to loss of Gibbsianness. We leave this for future work.
\end{enumerate}

\begin{acknowledgements}
	This manuscript is part of the author's thesis prepared at the University of Victoria. I am grateful to my supervisors Chris Bose and Anthony Quas for many helpful discussions and comments on versions of this manuscript.
\end{acknowledgements}

\pagebreak

\section{Appendix: Cones in $C(\S_{A}^{+})$}
\begin{definition}
	Let $V$ be a real vector space a subset $\c \subseteq V$ is called a cone if 
	\begin{enumerate}
		\item
		$a\c  =\c$ for all $a >0$.
		
		\item
		$\c$ is convex.
		
		\item 
		$\c \cap (- \c) = \emptyset$
	\end{enumerate}
	if $\c$ satisfies only $(1)$ and $(2)$ we say that $\c$ is a wedge. We say that $\c$ is a closed cone if $\c \cup\set{0}$ is closed.
\end{definition}

Every cones induces a partial ordering on the vector space $V$ by
\[ x \leq_{\c} y \iff y-x \in \c \cup\set{0} \] 
for $x\in \c$ and $y \in V$ we say that $x$ \emph{dominates} $y$ if there exists $\alpha, \beta$ such that
\[ \alpha x \leq_{\c} y \leq_{\c} \beta x \]
if $x$ dominates $y$ and $y$ dominates $x$ then $x,y$ are said to be \emph{comparable}. A cone $\c$ is called \emph{almost Archimedean} if whenever $x \in \c$ and $y \in V$ are such that $-\e x \leq_{\c} y \leq_{\c} \e x$ for all $\e>0$ it follows that $y=0$. If $x\geq_{\c}0$ dominates $y$ then define
\begin{gather*}
m(y/x; \c)=\sup\set{\alpha : \alpha x \leq_{\c} y}\\
M(y/x; \c)=\inf \set{\beta:y \leq_{\c} \beta x}
\end{gather*}
if $x$ and $y$ are comparable we define 
\[ \Theta_{\c}(x, y)=\log\left(\frac{M(x/y;\c)}{m(x/y;\c)}\right).  \]

This is known as Hilbert's metric and has been used in a number of different contexts. In dynamics it is most often used to prove upper bounds on the rate of convergence for transfer operators. This function $\Theta_{\c}$ is a projective psuedo metric in the following sense.

\begin{proposition}
	Let $V$ be a real vector space and $\c \subseteq V$ a cone. Suppose that $x,y,z$ are comparable and $a,b >0$ then
	\begin{enumerate}
		\item 
		$\Theta_{\c}(x,y)\geq 0$.
		
		\item 
		$\Theta_{\c}(x,y)=\Theta_{\c}(y,x)$.
		
		\item 
		$\Theta_{\c}(x,y)\leq \Theta_{\c}(x,z)+\Theta_{\c}(z,y)$.
		
		\item 
		$\Theta_{\c}(ax,by)=\Theta_{\c}(x,y)$.
	\end{enumerate}
	If $\c$ is almost Archemedean then
	\[ \Theta_{\c}(x,y)=0 \implies x= a y \]
	for some $a> 0$.
\end{proposition}
\begin{proof}
	These follow by direct computation. See for example \cite[lemma 2.7]{MR1297895}.
\end{proof}

The true utility of Hilbert's metric is the following theorem.

\begin{theorem}\label{BirkhoffContraction}
	(Birkhoff \cite{BirkhoffHilbertMetric}) Let $\c_{1},\c_{2}$ be closed cones and $L:V_{1} \to V_{2}$ a linear map such that $L\c_{1} \subseteq \c_{2}$. Then for all $\phi , \psi \in \c_{1}$
	\[ \Theta_{\c_{2}}(L\phi , L\psi) \leq \tanh\left(\frac{\diam_{\c_{2}}(L\c_{1})}{4}\right)\Theta_{\c_{1}}(\phi , \psi) \]
	where 
	\[ \diam_{\c_{2}}(L\c_{1})=\sup\set{\Theta_{\c_{2}}(f,g ):f,g\in L\c_{1}} \]
	and $\tanh \infty =1$.
\end{theorem}

Given a closed cone $\c$ in a Banach space $X$ we can define the dual
\[ \c^{\ast}=\set{\psi \in X^{\ast}: \inn{x, \psi}\geq 0 \text{ for all } x \in \c} \]
Notice that if $L\c_{1}\subseteq \c_{2}$ then $L^{\ast}\c_{2}\subseteq \c_{1}^{\ast}$. To see this suppose that $\psi \in \c_{2}^{\ast}$ and $x \in \c_{1}$ then
\[ \inn{x, L^{\ast}\psi}= \inn{Lx,\psi} \geq 0. \]
A word of caution: in general $\c^{\ast}$ is not a cone but a wedge.

\begin{proposition}\label{DualconeMetricComputation}
	Let $\c$ be a closed cone and $x,y \in \c$ such that $x$ and $y$ are comparable. Then for any $\phi \in \c^{\ast}$, $\inn{x,\phi}=0$ if and only if $\inn{y,\phi}=0$ and
	\[ \Theta_\c(x,y) = \log \left(\sup\set{\frac{\inn{x, \phi}\inn{y,\psi}}{\inn{y,\phi}\inn{x,\psi}}:\psi,\phi \in \c^{\ast}\text{ and }\inn{y,\phi}\inn{x,\psi}\neq 0}\right)\]
	\begin{proof}
		The proof can be found in \cite[Lemma 1.4]{eveson_nussbaum_1995}.
	\end{proof}
\end{proposition}

Let $\S_{A}^{+}$ be a topological Markov shift over a countable alphabet (we will work in this generality so that the industrious reader can convince themselves that under the appropriate assumptions the results about factors can be carried over to countable state spaces). When $\Sigma_{A}^{+}$ is a shift space it is often nice to think of $C_{b}(\Sigma_{A})$ as a ``$\ell^{\infty}$'' direct sum of the spaces
\[ C([a]) :=\set{f \in C_{b}(\S_{A}^{+}):f(x)=0 \text{ for all }x \notin [a]} \] 
where $a$ is a symbol in the alphabet. Recall that given a collection of Banach spaces $\set{X_{i}}_{i =0}^{\infty}$ we can define
\[ \bigoplus_{i=0}^{
	\infty}X_{i}=\set{(a_{i})_{i =0}^{\infty}:\sup_{i \geq 0}\norm{a_{i}}_{X_{i}}<\infty} \]
and 
\[ \norm{(a_{i})_{i=0}^{\infty}}=\sup_{i \geq 0}\norm{a_{i}}_{X_{i}}. \]
This along with coordinate-wise addition and scalar multiplication is one of a number of possible ways to take the direct sum of Banach spaces. 
\begin{proposition}
	The function $\norm{(a_{i})_{i=0}^{\infty}}=\sup_{i \geq 0}\norm{a_{i}}_{X_{i}}$ defines a norm on $\bigoplus_{i=0}^{\infty}X_{i}$ which makes $\bigoplus_{i =0}^{\infty}X_{i}$ into a Banach space.
\end{proposition}
It is not hard to see that 
\[ C_{b}(\S_{A}^{+})\cong \bigoplus_{a\in \S}C([a]). \] 
Moreover given a subset of symbols $S$ we identify
\[ \bigoplus_{a \in S}C([a])= \set{f \in C_{b}(\S_{A}):f(x)=0 \text{ for all }x \notin \bigcup_{a \in S}[a]}. \]
We will also be concerned with cones which arise in the following way.

\begin{proposition}\label{directsumofcones}
	Suppose that $\set{X_{i}}_{i=0}^{\infty}$ is a set of Banach spaces and $\c_{i} \subseteq X_{i}$ are closed cones. Then 
	\[ \bigoplus_{i=0}^{\infty}\c_{i}:=\set{(x_{i})_{i=0}^{\infty}\in \bigoplus_{i=0}^{\infty}X_{i}:x_{i}\in \c_{i} \cup \set{0}}\setminus \set{(0)_{i=0}^{\infty}} \]
	is a closed cone in $\bigoplus_{i=0}^{\infty}X_{i}$.
\end{proposition}
\begin{proof}
	As addition and scalar multiplication are coordinate wise it is clear that $\bigoplus_{i=0}^{\infty}\c_{i}$ is a cone. To see that it is closed suppose that $x$ is in the closure of $\bigoplus_{i=0}^{\infty}\c_{i}$ and take a net $x^{\alpha}$ converging to $x$. By the definition of the norm we have that for any $i$ $x^{\alpha}_{i}$ converges to $x_{i}$, as $\c_{i}\cup \set{0}$ is closed this implies that $x_{i} \in \c_{i}\cup\set{0}$. Therefore $x \in \bigoplus_{i=0}^{\infty}\c_{i} \cup \set{0}$.
\end{proof}

\begin{lemma}
	Suppose that $\set{X_{i}}_{i=0}^{\infty}$ is a set of Banach spaces and $\c_{i} \subseteq X_{i}$ are closed cones. Then 
	\[ m\left(x/y;\bigoplus_{i} \c_{i}\right)=\inf_{i}m(x_{i}/y_{i};\c_{i}) \]
	and
	\[ M\left(x/y; \bigoplus_{i} \c_{i}\right)=\sup_{i}M(x_{i}/y_{i}; \c_{i}). \]
	Thus
	\[ \Theta_{\bigoplus_{i} \c_{i}}(x,y)=\log \left(\sup_{i,j}\frac{M(x_{i}/y_{i}; \c_{i})}{m(x_{j}/y_{j};\c_{j})}\right) \]
\end{lemma}
\begin{proof}
	We will show this for $m$ the proof for $M$ is similar. Notice that 
	\begin{align*}
	m\left(x/y;\bigoplus_{i} \c_{i}\right)&=\sup\set{\alpha : a x_{i}\leq_{\c_{i}}y_{i} \text{ for all }i }\\
	&=\sup \bigcap_{i}\set{\alpha : \alpha x_{i}\leq_{\c_{i}} y_{i}} \\
	&= \sup \bigcap_{i}[0, m(x_{i}/y_{i};\c_{i})] \\
	&= \inf_{i} m(x_{i}/y_{i};\c_{i})
	\end{align*}
\end{proof}

For a metric space $(Z,d)$ define the cones
\[ C_{b}(Z)^{+}=\set{f\in C_{b}(Z):f \geq 0} \]
and
\[ \c(Z,d):=\set{f \in C_{b}(Z):f>0 \text{ and }f(x)\leq e^{d(x,y)}f(y) \text{ for all }x,y\in Z}. \]

\begin{lemma}
	Assume that $Z$ has no isolated points. For $f,g\in \c(Z,d)$ we have
	\[ m(g/f;\c(Z,d))=\inf_{x,y\in Z, x \neq y}\frac{e^{d(x,y)}g(y)-g(x)}{e^{d(x,y)}f(y)-f(x)} \]
	and
	\[ M(g/f;\c(Z,d))=\sup_{x,y\in Z, x \neq y}\frac{e^{d(x,y)}g(y)-g(x)}{e^{d(x,y)}f(y)-f(x)}. \]
\end{lemma}
\begin{proof}
	This is a mild generalization of the proof of lemma 2.2 in \cite{LiveraniDOC} we provide the details for the reader's convenience. We will compute $m(f/g;\c(Z,d))$, $M(f/g:\c(Z,d))$ is analogous. Notice $\alpha f \leq _{\c(Z,d)}g$ if and only if
	\[ \alpha \leq \frac{g(x)}{f(x)} \forall x \in X \text{ and }\alpha \leq \frac{e^{d(x,y)}g(y)-g(x)}{e^{d(x,y)}f(y)-f(x)} \forall x \neq y.  \]
	Thus 
	\begin{align*}
	m(g/f;\c(Z,d))&= \sup \set{\alpha: \alpha f \leq_{\c(Z,d)}g}\\
	&=\min \set{\inf_{x \in Z}\frac{g(x)}{f(x)}, \inf_{x,y \in Z,x \neq y} \frac{e^{d(x,y)}g(y)-g(x)}{e^{d(x,y)}f(y)-f(x)} }.
	\end{align*}
	So the claim is that 
	\[  \inf_{x,y \in Z,x \neq y} \frac{e^{d(x,y)}g(y)-g(x)}{e^{d(x,y)}f(y)-f(x)}\leq \inf_{x \in Z}\frac{g(x)}{f(x)}.\]
	Let $z \in Z$ and take $x\neq z$ such that $\frac{g(x)f(z)}{f(x)g(z)}\leq 1$. Note this can always be done unless $\frac{g}{f}$ has a unique minimum at $z$. Then
	\begin{align*}
	\inf_{x,y \in Z,x \neq y} \frac{e^{d(x,y)}g(y)-g(x)}{e^{d(x,y)}f(y)-f(x)} &\leq \frac{e^{d(x,z)}g(x)-g(z)}{e^{d(x,z)}f(x)-f(z)}\\
	&= \frac{e^{d(x,z)}\frac{g(x)}{f(x)}f(x)-\frac{g(z)}{f(z)}f(z)}{e^{d(x,z)}f(x)-f(z)}\\
	&= \frac{g(z)}{f(z)}\frac{e^{d(x,z)}\frac{g(x)f(z)}{f(x)g(z)}f(x)-f(z)}{e^{d(x,z)}f(x)-f(z)}\\
	&\leq \frac{g(z)}{f(z)} \text{ as }\frac{g(x)f(z)}{f(x)g(z)}\leq 1.
	\end{align*}
	If $z$ is the unique minimum of $\frac{g}{f}$ then take a sequence which is not eventually $z$ converging to $z$ and the same inequality follows (hence the assumption that $Z$ contains no isolated points).
\end{proof}

We are concerned with the case that $Z$ is a cylinder set in $\S_{A}^{+}$ and $d$ some metric. Given a metric and a number $\sigma>0$ we can define a new metric by the function $(x,y)\mapsto \sigma d(x,y)$, the new metric is called $\sigma d$. The next corollary will be one of our fundamental tools, it has been noticed previously, see for example \cite{MR2083432} and \cite{LiveraniDOC}, although we state it in a somewhat different way.

\begin{corollary}\label{ThetaBoundRegularityCones}
	Suppose that $\set{(Z_{i},d_{i})}_{i}$ is a countable collection of metric spaces without isolated points. If $f,g \in \bigoplus_{i}\c(Z_{i},\sigma d_{i})$ then 
	\[ \Theta_{\bigoplus_{i}\c(Z_{i},d_{i})}(f,g)\leq 2 \log \frac{1+\sigma}{1-\sigma}+\Theta_{\bigoplus_{i}C(Z_{i})^{+}}(f,g) \]
\end{corollary}
\begin{proof}
	Notice that for and $i,j$ we have that 
	\begin{align*}
	&\frac{M(f_{i}/g_{i};\c(Z_{i},d))}{m(f_{j}/g_{j};\c(Z_{j},d))}=\\
	&\sup\set{\frac{e^{d(x,y)}g(y)-g(x)}{e^{d(x,y)}f(y)-f(x)} \cdot \frac{e^{d(w,z)}f(z)-f(w)}{e^{d(w,z)}g(z)-g(w)}:x,y \in Z_{i}, w,z \in Z_{j}, x \neq y, w \neq z}
	\end{align*}
	and
	\begin{align*}
	&\frac{e^{d(x,y)}g_{i}(y)-g_{i}(x)}{e^{d(x,y)}f_{i}(y)-f_{i}(x)} \cdot \frac{e^{d(w,z)}f_{j}(z)-f_{j}(w)}{e^{d(w,z)}g_{j}(z)-g_{j}(w)}\\
	&= \frac{e^{d(x,y)}-\frac{g_{i}(x)}{g_{i}(y)}}{e^{d(x,y)}-\frac{f_{i}(x)}{f_{i}(y)}} \cdot \frac{e^{d(w,z)}-\frac{f_{j}(w)}{f_{j}(z)}}{e^{d(w,z)}-\frac{g_{j}(w)}{g_{j}(z)}}\cdot \frac{g_{i}(y)f_{j}(z)}{f_{i}(y)g_{j}(z)}\\
	&\leq \frac{e^{d(x,y)}-e^{-\sigma d(x,y)}}{e^{d(x,y)}-e^{\sigma d(x,y)}} \cdot \frac{e^{d(w,z)}-e^{-\sigma d(w,z)}}{e^{d(w,z)}-e^{\sigma d(w,z)}}\cdot \frac{g_{i}(y)f_{j}(z)}{f_{i}(y)g_{j}(z)}\\
	&=\frac{e^{d(x,y)}(1-e^{-(1+\sigma) d(x,y)})}{e^{d(x,y)}(1-e^{-(1-\sigma) d(x,y)})} \cdot \frac{e^{d(w,z)}(1-e^{-(1+\sigma) d(w,z)})}{e^{d(w,z)}(1-e^{-(1-\sigma) d(w,z)})}\cdot \frac{g_{i}(y)f_{j}(z)}{f_{i}(y)g_{j}(z)}\\
	&=\frac{1-e^{-(1+\sigma) d(x,y)}}{1-e^{-(1-\sigma) d(x,y)}} \cdot \frac{1-e^{-(1+\sigma) d(w,z)}}{1-e^{-(1-\sigma) d(w,z)}}\cdot \frac{g_{i}(y)f_{j}(z)}{f_{i}(y)g_{j}(z)}\\
	&\leq \frac{(1+\sigma) d(x,y)}{(1-\sigma) d(x,y)} \cdot \frac{(1+\sigma) d(w,z)}{(1-\sigma) d(w,z)}\cdot \frac{g_{i}(y)f_{j}(z)}{f_{i}(y)g_{j}(z)}\\
	&=  \frac{(1+\sigma)}{(1-\sigma)} \cdot \frac{(1+\sigma)}{(1-\sigma)}\cdot \frac{g_{i}(y)f_{j}(z)}{f_{i}(y)g_{j}(z)}
	\end{align*}
	where in line 6 we have used that for $x \geq y >0$
	\[ \frac{1-e^{-x}}{1-e^{-y}}\leq \frac{x}{y}. \]
\end{proof}

\bibliography{FactorsBib}{}

\begin{thebibliography}{10}

\bibitem{BirkhoffHilbertMetric}
G.~Birkhoff.
\newblock Extensions of {J}entzsch's theorem.
\newblock {\em Trans. Amer. Math. Soc.}, 85:219--227, 1957.

\bibitem{bowen1975equilibrium}
R.~Bowen.
\newblock {\em Equilibrium states and the ergodic theory of {A}nosov
  diffeomorphisms}, volume 470 of {\em Lecture Notes in Mathematics}.
\newblock Springer-Verlag, Berlin, revised edition, 2008.
\newblock With a preface by David Ruelle, Edited by Jean-Ren\'e Chazottes.

\bibitem{chazottes2003projection}
J.-R. Chazottes and E.~Ugalde.
\newblock Projection of {M}arkov measures may be {G}ibbsian.
\newblock {\em J. Statist. Phys.}, 111(5-6):1245--1272, 2003.

\bibitem{chazottes2009preservation}
J.-R. Chazottes and E.~Ugalde.
\newblock On the preservation of {G}ibbsianness under symbol amalgamation.
\newblock In {\em Entropy of hidden {M}arkov processes and connections to
  dynamical systems}, volume 385 of {\em London Math. Soc. Lecture Note Ser.},
  pages 72--97. Cambridge Univ. Press, Cambridge, 2011.

\bibitem{eveson_nussbaum_1995}
S.~P. Eveson and R.~D. Nussbaum.
\newblock Applications of the {B}irkhoff-{H}opf theorem to the spectral theory
  of positive linear operators.
\newblock {\em Math. Proc. Cambridge Philos. Soc.}, 117(3):491--512, 1995.

\bibitem{MR1297895}
S.~P. Eveson and R.~D. Nussbaum.
\newblock An elementary proof of the {B}irkhoff-{H}opf theorem.
\newblock {\em Math. Proc. Cambridge Philos. Soc.}, 117(1):31--55, 1995.

\bibitem{johansson_öberg_pollicott_2017}
A.~Johansson, A.~\"Oberg, and M.~Pollicott.
\newblock Phase transitions in long-range ising models and an optimal condition
  for factors of $g$ -measures.
\newblock {\em Ergodic Theory and Dynamical Systems}, page 1–14, 2017.

\bibitem{kempton2011factors}
T.~M.~W. Kempton.
\newblock Factors of {G}ibbs measures for subshifts of finite type.
\newblock {\em Bull. Lond. Math. Soc.}, 43(4):751--764, 2011.

\bibitem{kondah1997vitesse}
A.~Kondah, V.~Maume, and B.~Schmitt.
\newblock Vitesse de convergence vers l'\'etat d'\'equilibre pour des
  dynamiques markoviennes non h\"old\'eriennes.
\newblock {\em Ann. Inst. H. Poincar\'e Probab. Statist.}, 33(6):675--695,
  1997.

\bibitem{LiveraniDOC}
C.~Liverani.
\newblock Decay of correlations.
\newblock {\em Ann. of Math. (2)}, 142(2):239--301, 1995.

\bibitem{MR2083432}
F.~Naud.
\newblock Birkhoff cones, symbolic dynamics and spectrum of transfer operators.
\newblock {\em Discrete Contin. Dyn. Syst.}, 11(2-3):581--598, 2004.

\bibitem{HolderGibbsStates}
M.~Piraino.
\newblock Projections of {G}ibbs states for {H}{\"o}lder potentials.
\newblock {\em J. Stat. Phys.}, 170(5):952--961, Mar 2018.

\bibitem{pollicott2011factors}
M.~Pollicott and T.~Kempton.
\newblock Factors of {G}ibbs measures for full shifts.
\newblock In {\em Entropy of hidden {M}arkov processes and connections to
  dynamical systems}, volume 385 of {\em London Math. Soc. Lecture Note Ser.},
  pages 246--257. Cambridge Univ. Press, Cambridge, 2011.

\bibitem{MR2853610}
E.~Verbitskiy.
\newblock On factors of {$g$}-measures.
\newblock {\em Indag. Math. (N.S.)}, 22(3-4):315--329, 2011.

\bibitem{MR1783787}
P.~Walters.
\newblock Convergence of the {R}uelle operator for a function satisfying
  {B}owen's condition.
\newblock {\em Trans. Amer. Math. Soc.}, 353(1):327--347, 2001.

\bibitem{Yoo2010}
J.~Yoo.
\newblock On factor maps that send {M}arkov measures to {G}ibbs measures.
\newblock {\em J. Stat. Phys.}, 141(6):1055--1070, 2010.

\end{thebibliography}
\bibliographystyle{abbrv}

\end{document}